\UseRawInputEncoding
\documentclass[11pt]{article}

\usepackage{amsthm}
\usepackage{graphicx}
\usepackage{float}

\usepackage[abbrev]{amsrefs}
\usepackage{newtxmath}
\usepackage{caption}

\usepackage{amsmath}
\usepackage{mathtools}
\mathtoolsset{showonlyrefs}

\usepackage{pb-diagram}
\usepackage[all]{xy} 
\makeindex

\newtheorem{thm}{Theorem}[section]

\newtheorem{prp}{Proposition}[section]
\newtheorem{lem}{Lemma}[section]

\newtheorem{ex}{Example}[section]
\newtheorem{rem}{Remark}[section]

\newcommand{\R}{\mathbb{R}}
\newcommand{\Z}{\mathbb{Z}}
\newcommand{\mE}{\mathbb{E}}
\newcommand{\mP}{\mathbb{P}}

\newcommand{\bbL}{\mathbb{L}}

\newcommand{\mT}{\mathcal{T}}

\newcommand{\mB}{\mathcal B}
\newcommand{\wt}{\widetilde}

\newcommand{\V}{\text{Var}}

\title{Covering monotonicity of the limit shapes of \\first passage percolation on crystal lattices}
\author{Tatsuya Mikami\thanks{Department of Mathematics, Kyoto University, E-mail:mikami.tatsuya.68z@st.kyoto-u.ac.jp ）}}

\date{}

\begin{document}
\maketitle

\abstract{
This paper studies the first passage percolation (FPP) model: each edge in the cubic lattice $\bbL^d$ is assigned a random passage time, and consideration is given to the behavior of the percolation region $B(t)$, which consists of those vertices that can be reached from the origin within a time $t > 0$. Cox and Durrett showed the shape theorem for the percolation region, saying that the normalized region $B(t)/t$ converges to some limit shape $\mathcal{B}$. This paper introduces a general FPP model defined on crystal lattices, and shows the monotonicity of the limit shapes under covering maps, thereby providing insight into the limit shape of the cubic FPP model. }

\section{Introduction}
\label{sec:introduction}

\subsection{Background}
\label{subsec:background}
Percolation theory is a branch of probability theory that describes the behavior of clusters of randomly obtained objects. One of the most famous percolation models is the \emph{bond percolation model}, in which each edge (bond) of an infinite, connected, and locally finite graph $X = (V, E)$ is assumed to be open with the same probability $p \in [0, 1]$, independently of all other edges. This model has been of great interest regarding the \emph{critical probability} $p_c(X) \in [0, 1]$, which is the value such that all clusters (connected components) are finite when $p < p_c(X)$ and there exists an infinite cluster when $p>p_c(X)$. This model has been mostly studied in the $d$-dimensional cubic lattice $\bbL^d = (\Z^d, \mE^d)$, where $\Z^d$ is the set of all $d$-tuples $x = (x_1, \ldots, x_d)$ of integers $x_i$, and the edge set $\mE^d$ is the set of all unordered pairs $\{x, y\}$ of $\Z^d$ with $\|x - y\|_1 = 1$ (here, $\|\cdot\|_p$ represents the $L_p$-norm on $\R^d$). \par
The paper \cite{beyond_Z^d} proposes a comprehensive study of percolation in a more general context than $\bbL^d$, and raises the ``covering'' problem formulated as follows. Let $G \curvearrowright X$ be a free group action on a graph $X$. The quotient graph $X_1:= X/G$ is defined as the graph whose vertices are $G$-orbits, and an edge $\{Gx, Gy\}$ appears in $X_1$ if there are representatives $x_0 \in Gx$, $y_0 \in Gy$ that are neighbors in $X$. In this setting, the comparison
\begin{equation}
\label{eq:monotonicity_critical}
	p_c(X) \leq p_c(X_1)
\end{equation}
of two critical probabilities holds, and the strictness of \eqref{eq:monotonicity_critical} is given by~\cite{Strict_monotonicity} under some assumptions. This result implies that the cluster in a graph tends to be larger than that in its quotient graph. These studies include the study in \cite{Campanino}, which shows that the critical probability $p_c(\mathbb{L}^3)$ is strictly less than that $p_c(\mathbb{T})$ of the triangular lattice $\mathbb{T}$.
\par
The aim of this paper is to show an analogue of \eqref{eq:monotonicity_critical} for the \emph{first passage percolation} (FPP) model. The FPP model, which was introduced in 1965 by Hammersley and Welsh~\cite{Hammersley}, is a time evolution version of the bond percolation model: each edge $e \in \mE^d$ of the $d$-dimensional cubic lattice is independently assigned a random nonnegative time $t_e \geq 0$ according to a fixed distribution $\nu$. The passage time $T(\gamma)$ of a path $\gamma = (e_1,\ldots, e_r)$ is defined as the sum $T(\gamma) := \sum_{i=1}^r t_{e_i}$.   
For two points $x$, $y \in \R^d$, we denote by $T(x, y)$ the first passage time from $x$ to $y$, that is,
\begin{equation*}
	T(x, y) := \inf\{T(\gamma) \,:\, \gamma \text{ is a path from $x^{\prime}$ to $y^{\prime}$} \},
\end{equation*}
where $x^{\prime}$, $y^{\prime} \in \Z^d$ are the closest lattice points of $x$ and $y$, respectively. The percolation region $B(t)$ is defined as
\begin{equation*}
	B(t) := \{x \in \R^d \,:\, T(0, x)\leq t\}
\end{equation*}
for a time $t > 0$. Cox and Durrett~\cite{Cox_Durrett} showed the \emph{shape theorem}, a ``law of large numbers'' for the percolation region: if the time distribution $\nu$ satisfies the moment condition 
\begin{equation}
\label{eq:a_moment_condition}
	\mE\min\{t_1,\ldots, t_{2d} \}^{d} < \infty, 
\end{equation}
where random variables $t_1, \ldots, t_{2d} \sim \nu$ are independent copies of random times, then the normalized region $B(t)/t$ converges to some limit shape $\mB$, which may coincide with the whole space $\R^d$, as $t \rightarrow \infty$. 

\subsection{Main result}
This paper presents an analogue of \eqref{eq:monotonicity_critical} for the limit shape $\mB$ in the framework of a crystal lattice-model. First, we consider a general FPP model defined on \emph{crystal lattices}, using the formulation of discrete geometric analysis that was first introduced by Kotani and Sunada~\cite{Standard_Realizations}. Here, a $d$-dimensional crystal lattice is a regular covering graph $X$ over a finite graph $X_0$ whose covering transformation group $L$ is a free abelian group with rank $d$. We consider its ``shape''  as a periodic realization $(\Phi, \rho)$, where $\Phi:X\rightarrow \R^d$ is a map into the Euclidian space $\R^d$, and the homomorphism $\rho: L \rightarrow \R^d$ represents the period of the realization (see Sect.~\ref{subsec:crystal_lattices}). In this setting, we first give the relation between the limit shape $\mB$ and a realization of a crystal lattice, and we give the symmetric properties of $\mB$ derived from those of the realized crystal $\Phi(X)$. \par
The main result in this paper is formulated as follows. Let $\Phi:X \rightarrow \R^d$ be a periodic realization of a $d$-dimensional crystal lattice $X$. By identifying some $d_1$-dimensional subspace of $\R^d$ with $\R^{d_1}$, we shall observe that the orthogonal projection $P(\Phi(X))$ onto $\R^{d_1}$ coincides with the image of a periodic realization $\Phi_1: X_1 \rightarrow \R^{d_1}$ of a $d_1$-dimensional crystal lattice $X_1$. We also see that $X_1$ can be written by $X/G$ for some group action $G \curvearrowright X$, and that the following commutative property holds:
\begin{equation}
\label{diagram:projection}
\xymatrix{
X\ar[r]^-{\Phi}\ar[d]_-{\omega}\ar@{}[rd]|{\circlearrowright}& \R^d\ar[d]^-{P}\\
X_1\ar[r]_-{\Phi_1}&\R^{d_1},  
}
\end{equation}
where $\omega: X \rightarrow X_1$ is the quotient map and $P: \R^d \rightarrow \R^{d_1}$ is the orthogonal projection (see~\cite[Sect. 7.2]{Topological_Crystallography}). In this setting, we consider the independent FPP models on $X$ and $X_1$ with the same distribution $\nu$ satisfying an analogue of the moment condition \eqref{eq:a_moment_condition} for $X$ and $X_1$. We show that the limit shape $\mB_1$ of $X_1$ is included in the projection $P(\mB)$ of the limit shape $\mB$ of $X$. 
\begin{thm}
\rm
\label{thm:covering_monotonicity}
Let $\Phi:X \rightarrow \R^d$ and $\Phi_1:X_1 \rightarrow \R^{d_1}$ be periodic realizations of crystal lattices $X$, $X_1$ satisfying the projective relation \eqref{diagram:projection}. Then, for the limit shapes $\mB$, $\mB_1$ of $X$, $X_1$, the following holds: 
\begin{align*}
	\mB_1 \subset P(\mB).
\end{align*}
\end{thm}
This result gives insights regarding the limit shape of the cubic FPP model. Namely, in order to observe the shape of $\mB$ of the cubic lattice $\bbL^d$, we can consider the projection $P:\R^d \rightarrow \R^{d_1}$ in some suitable direction and obtain a periodic realization of a covered lattice $X_1$. Then we can obtain that $\mB_1 \subset P(\mB)$ for the limit shape $\mB_1$ of $X_1$, implying that the projection of the limit shape $\mB$ of $\mathbb{L}^d$ to the the space $\R^{d_1}$ is bounded below by the limit shape $\mB_1$ of $X_1$. 
\par
We remark on several works related to this formulation. First, the paper~\cite{Multidimensional_lattice} gives a formulation of the ``periodic graph'' by group action on a graph, and it is equivalent to the definition of crystal lattices in this paper (see Sect.~\ref{subsec:covering_graphs}). It follows from~\cite[Theorem~7.2]{Topological_Crystallography} that our formulation of crystal lattices is also essentially equivalent to that in \cite[Sect. 2.1]{Kesten_textbook}, which defines a ``periodic graph'' as a graph imbedded in $\R^d$ in such a way that each coordinate vector of $\R^d$ is a period for the image. The paper \cite{1_dFPP} studies the asymptotic properties of the FPP model on a $1$-dimensional graph, which corresponds to the crystal lattice with dimension $1$ in this paper. We also note that \cite{FPPonTriangular2016, FPPonTriangular2019} study the critical FPP model on the triangular lattice. 
\par
The study of crystal lattices is rich in considerations of the shape of lattices, as exemplified by the concept of the \emph{standard realization}, which is the realization with maximal symmetry among all realizations (see \cite{Topological_Crystallography}). The cubic lattice that are often considered in the percolation model, for example, are included in the framework of the standard realization. Thus, the formulation of a periodic FPP model in this paper is suitable for studying the relationship between the shape of the percolation region and the lattices. 
\par
The remainder of this paper is organized as follows. In Sect.~\ref{sec:Preliminaries}, we review the concept of crystal lattices and basic properties. In Sect.~\ref{sec:setting_and_shape_theorem}, we formulate the FPP model on a general crystal lattice and give a generalized shape theorem (Theorem~\ref{thm:shape_theorem}). In Sect.~\ref{sec:covering_monotonicity}, we give the proof of the main theorem (Theorem~\ref{thm:covering_monotonicity}). 

\section{Preliminaries}
\label{sec:Preliminaries}
\subsection{Covering graphs}
\label{subsec:covering_graphs}
A crystal lattice, which is the main object in this paper, is defined as a covering graph over a finite graph. First, in this subsection, we review the concept of covering graphs. We refer to \cite{Topological_Crystallography} for more detailed description. \par
A \emph{graph} is an ordered pair $X=(V, E)$ of disjoint sets $V$ and $E$ with two maps $i:E \rightarrow V \times V$, $\iota:E \rightarrow E$ satisfying 
\begin{align*}
	&\iota^2 = I_E \ (\text{the identity map of $E$)}  \text{, and}\\
	&\iota(e) \neq e, \  i(\iota(e)) = \tau(i(e))
\end{align*}
for any $e\in E$, where $\tau: V \times V \rightarrow V \times V$ is the map defined by $\tau(x, y) = (y, x)$. We call $i$ and $\iota$ the \emph{incident map} and the \emph{inversion map} of $X$, respectively. We put $i(e) = (o(e), t(e))$ and call $o(e)$ and $t(e)$ the \emph{origin} and the \emph{terminus}, respectively. $\iota(e)$ is called the \emph{inversion} of $e$ and is sometimes written as $\bar e$. For $x \in V$, we denote by $E_x:= \{e \in E\,:\, o(e) = x\}$. \par
For two graphs $X_1 = (V_1, E_1)$ and $X_2 = (V_2, E_2)$, a \emph{morphism} $f: X_1 \rightarrow X_2$ is a pair $f = (f_V, f_E)$ of two maps $f_V: V_1 \rightarrow V_2$, $f_E: E_1 \rightarrow E_2$ satisfying
\begin{align*}
	&i(f_E(e)) = (f_V(o(e)), f_V(t(e))),\\
    &f_E(\bar e) = \overline{f_E(e)}.
\end{align*}
When both $f_V$ and $f_E$ are bijective, the morphism $f$ is called an \emph{isomorphism}. We abbreviate $f_V$ and $f_E$ by $f$ when there is no confusion. For a graph $X$, we denote by $\text{Aut}(X)$ the automorphism group of $X$. \par
An \emph{action} $G \curvearrowright X$ of a group $G$ on a graph $X$ is a group homomorphism $h: G \rightarrow \text{Aut}(X)$, which naturally gives rise to actions $G\curvearrowright V$ and $G\curvearrowright E$ by $gx:=h(g)(x)$ for $x \in V$ and $ge:=h(g)(e)$ for $e \in E$, respectively. We say $G$ acts on $X$ \emph{freely} when the action $G \curvearrowright V$ is free and $ge \neq \bar e$ for any $g \in G$ and $e \in E$. For a free action $G \curvearrowright X$, we define the \emph{quotient graph} $X/G$ of $X$ as the pair $X/G = (V/G, E/G)$ of the orbit spaces $V/G$ and $E/G$ whose incident and inversion maps are induced from those of $X$. 
\begin{rem}
\rm
	This formulation of quotient graphs is more inclusive than the one we introduced in Sect \ref{subsec:background}, in the sense that it allows graphs with parallel edges and loops.   
\end{rem}
Taking a quotient can be characterized by a \emph{covering map}, namely, a morphism $\omega: X \rightarrow X_0$ from
a connected graph $X = (V, E)$ to $X_0 = (V_0, E_0)$ satisfying 
\begin{itemize}
	\item $\omega_V$ is surjective, and
\item for every $x \in V$, the restriction ${\omega_E}_{\restriction_{E_x}}: E_x \rightarrow E_{0, \omega(x)}$ is bijective.  
\end{itemize}
The graph $X$ is called a \emph{covering graph} of $X_0$. The \emph{covering transformation group} $G(\omega)$ of a covering map $\omega$ is the set of automorphisms $\sigma \in \text{Aut}(X)$ with $\omega \circ \sigma = \omega$.  We say that a covering map $\omega: X \longrightarrow X_0$ is \emph{regular} if for any $x$, $y \in V$ with $\omega(x) = \omega(y)$, there exists a transformation $\sigma \in G(\omega)$ such that $\sigma x = y$. \par
A \emph{path} $\gamma$ in a graph $X$ is a sequence $\gamma = (e_1, e_2, \ldots, e_r)$ of edges with $o(e_{i+1}) = t(e_i)$ for $i = 1,2,\ldots, r-1$. One of the most basic properties of a covering map $\omega: X \rightarrow X_0$ is the \emph{unique path-lifting property}: for any path $\gamma_0 = (e_{0,1}, e_{0,2}, \ldots, e_{0,r})$ in $X_0$ and a vertex $x \in X$ with $\omega(x) = o(e_{0,1})$, there exists a unique path $\gamma = (e_1, e_2, \ldots, e_r)$ in $X$, called a \emph{lifting} of $\gamma_0$, such that $o(e_1) = x$ and $\omega(\gamma) = \gamma_0$. 
\par
We can see that for a regular covering map $\omega: X \rightarrow X_0$, the action $G(\omega) \curvearrowright X$ of the covering transformation group $G(\omega)$ is free and its quotient graph $X/{G(\omega)}$ is isomorphic to $X_0$. On the other hand, the following theorem holds. 
\begin{thm}
\rm
(\cite{Topological_Crystallography}, Theorem~5.2)
\label{thm:alternaticve_defintion}
	Suppose a group $G$ acts freely on a graph $X$. Then the canonical projection $\omega: X \rightarrow X/G$ is a regular covering map whose covering transformation group is $G$. 
\end{thm}
 
\subsection{Crystal lattices}
\label{subsec:crystal_lattices}
A \emph{crystal lattice} $X$ is a regular covering graph over a finite graph $X_0$ whose transformation group $L$ is a free abelian group. The finite graph $X_0$ is called a \emph{base graph} of $X$ and the transformation group $L$ is called an \emph{abstract period lattice}. The \emph{dimension} $\dim X$ of a crystal lattice $X$ is defined to be the rank of $L$. \par
As we remarked in the previous subsection, we have an alternative description of a crystal lattice. Namely, a graph $X$ is a crystal lattice if and only if there exists a free action $L \curvearrowright X$ of a free abelian group $L$, and the quotient graph $X_0 := X/L$ is finite. The following are several examples of crystal lattices. 
\begin{ex}
\rm
	The $d$-dimensional cubic lattice $\mathbb{L}^d = (\Z^d, \mE^d)$, which we introduced in Sect.~\ref{subsec:background}, is a crystal lattice with dimension $d$. Indeed, the free abelian group $\Z^d$ acts naturally on $\mathbb{L}^d$ by translation, and the quotient graph is a bouquet (Fig.~\ref{fig:crystal_lattices}). 
\end{ex}

\begin{ex}
\rm
	The honeycomb lattice $\mathbb{H}$ and the triangular lattice $\mathbb{T}$ are also crystal lattices with dimension $2$. Indeed, we can give the action of $\Z^2$ by translation in such a way that the basis of $\Z^2$ comprises the translations shown in Fig.~\ref{fig:crystal_lattices}. \par
A diamond lattice is also a crystal lattice with dimension $3$, and it can be regarded as a higher-dimensional version of the honeycomb lattice in the sense that the quotient graph consists of two points and four parallel edges connecting them. 
\end{ex}

\begin{figure}[H]
\captionsetup{width=0.85\linewidth}
\centering
\includegraphics[width = 0.6\linewidth]{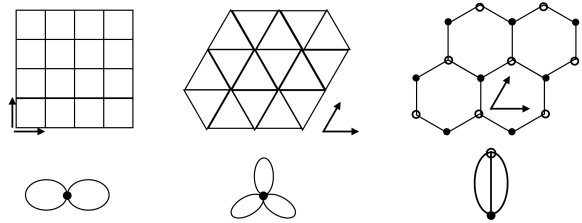}
 \caption{The cubic lattice (left), triangular lattice (center), and honeycomb lattice (right). The arrows indicate a basis of the action on each lattice, and the graphs below are their base graphs.}
\label{fig:crystal_lattices}
\end{figure}
Next, we formulate the ``shape'' of a crystal lattice as a map to the space $\R^d$. Let $X$ be a $d$-dimensional crystal lattice over a finite graph $X_0$, and let $L$ be its abstract period lattice. A \emph{realization} of $X$ into $\R^d$ is a map $\Phi: V \rightarrow \R^d$, where the edges of $X$ are realized as the segments connecting their endpoints. We often write a realization as $\Phi: X \rightarrow \R^d$, which is said to be \emph{periodic} if there exists an injective homomorphism $\rho: L \rightarrow \R^d$ satisfying the following conditions: 
\begin{itemize}
	\item the image $\Gamma := \rho(L)$ is a lattice group of $\R^d$; that is, there exists a basis $(a_1, \ldots, a_d)$ of $\R^d$ such that
\begin{equation*}
	\Gamma = \{\lambda_1 a_1 + \cdots + \lambda_d a_d \,:\, \lambda_i \in \Z\}; \text{ and}
\end{equation*}
\item for any vertex $x \in V$ and $\sigma \in L$,
\begin{align*}
	\Phi(\sigma x) = \Phi(x) + \rho(\sigma). 
\end{align*}
\end{itemize} 
Note that the realized crystal $\Phi(X)$ is invariant under the translation by any vector ${\bf b} \in \Gamma = \rho(L)$. The homomorphism $\rho:L \rightarrow \R^d$ is called the \emph{period homomorphism} of the realization $\Phi:X \rightarrow \R^d$. Though the period $\rho$ is determined uniquely from $\Phi$, we often write a periodic realization as the pair $(\Phi, \rho)$ in order to emphasize that $\rho$ represents the period of the realization.  
\begin{ex}
\rm
\label{ex:realizations}
Figure~\ref{fig:realizations} shows examples of periodic realizations of the honeycomb lattice. Here, the arrows show the basis of the lattice group. Note that the left and center ones have the same lattice group and thus the same period.  
\begin{figure}[H]
\captionsetup{width=0.85\linewidth}
\centering
\includegraphics[width = 10cm]{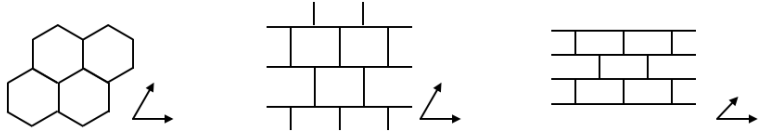}
 \caption{Three examples of periodic realizations of the honeycomb lattice.}
\label{fig:realizations}
\end{figure}	
\end{ex}

For a periodic realization $\Phi:X \rightarrow \R^d$ of a $d$-dimensional crystal lattice $X$ with a period homomorphism $\rho: L \rightarrow \R^d$, an orthogonal projection $P:\R^d \rightarrow \R^{d_1}$ onto some $d_1$-dimensional subspace $\R^{d_1}$ is said to be a \emph{rational projection} if the image $P(\rho(L))$ is a lattice group of $\R^{d_1}$. For this $P$, the quotient graph $X_1:= X/\text{Ker}(P\circ \rho)$ is a $d_1$-dimensional crystal lattice and the commutative diagram \eqref{diagram:projection} holds for some periodic realization $\Phi_1: X_1 \rightarrow \R^{d_1}$. Note that from Theorem~\ref{thm:alternaticve_defintion}, the quotient map $\omega: X \rightarrow X_1$ in \eqref{diagram:projection} is a regular covering map. 

\section{First passage percolation model and shape theorem}
\label{sec:setting_and_shape_theorem}

\subsection{Setting}
\label{subsec:settings}
Let $X = (V, E)$ be a $d$-dimensional crystal lattice over a finite graph $X_0$ whose abstract period lattice is $L$. In Sects~\ref{subsec:settings} and \ref{subsec:asymptotic_speed}, we fix a periodic realization $\Phi: X \rightarrow \R^d$ whose period homomorphism is $\rho: L \rightarrow \R^d$, and we set the lattice group $\Gamma := \rho(L)$. Fix a vertex $0 \in X$ as the ``origin'' of $X$ and suppose $\Phi(0) = 0 \in \R^d$. We also assume that
the periodic realization is \emph{nondegenerate}, that is, the map $\Phi: V \rightarrow \R^d$ is injective. 
Later we will remark that this assumption is actually not essential (Remark~\ref{rem:not_essential}). For the sake of brevity, we write $x \in X$ when $x \in V$, and for a point $\Phi(x) \in \R^d$, we write it by $x \in \R^d$ for short. \par
Throughout this paper, we deal with $X$ as an \emph{undirected graph}, whose edge set is given by the orbit space $E/ \Z_2$ of the action $\Z_2 := \Z/{2\Z} \curvearrowright E$ defined by $e \mapsto \bar e$. Here we simply denote by $E$ the set of orbit space $E/\Z_2$. Now we formulate the FPP model on $X$. Let $\Omega = [0, \infty)^E$ be the configuration space. Fix a \emph{time distribution} $\nu$, which is a probability measure on $[0, \infty)$. We define the probability measure $\mP$ on $\Omega$ as the product measure $\mP = \nu^{\otimes E}$. Let $\mE$ denote the expectation with respect to $\mP$. An element ${\bf t} = (t_e \,:\, e \in E) \in \Omega$ is called a \emph{configuration}. For a path $\gamma = (e_1,\ldots, e_r)$ in $X$, the passage time $T(\gamma)$ is the random variable defined as
\begin{equation*}
	T(\gamma) := \sum_{i=1}^r t_{e_i}.   
\end{equation*}
For two points $x$, $y \in \R^d$, we denote by $T(x, y)$ the first passage time between $x$ and $y$, that is, 
\begin{equation*}
	T(x, y) := \inf\{T(\gamma) \,:\, \gamma \text{ is a path from $x^{\prime}$ to $y^{\prime}$} \},
\end{equation*}
where $x^{\prime}$, $y^{\prime} \in X$ are the closest realized vertices of $x$ and $y$, respectively. The percolation region $B(t)$ is defined as
\begin{equation*}
	B(t) := \{x \in \R^d \,:\, T(0, x)\leq t\}
\end{equation*}
for a time $t > 0$. 
\begin{rem}
\label{rem:distribution_translation}
	We can easily see that $T(x, y)$ and $T(x+ {\bf b}, y + {\bf b})$ have the same distribution for any vector ${\bf b} \in \Gamma$. 
\end{rem}
Define the set of ``rational points''  $\mathcal{D}$ by 
\begin{align*}
	\mathcal{D} = \{q_1a_1+ \cdots q_da_d \,:\, q_i \in \mathbb{Q} \}, 
\end{align*}
where $(a_1, \ldots, a_d)$ is a basis of the lattice group $\Gamma$. Note that $\mathcal D$ coincides with the set of points $x \in \R^d$ such that $Nx \in \Gamma$ for some $N \in \mathbb{N}$, and thus $\mathcal D$ does not depend on the choice of a basis of $\Gamma$.  \par
 The \emph{edge connectivity} of $X$ is the minimum number $l_X \in \mathbb{N}$ such that there exists a set $\{e_1, \ldots, e_{l_X}\}$ of edges that separate $X$. Menger's theorem (see, e.g., \cite{Graph_theory}) gives an alternative description of edge connectivity as follows:  
\begin{align*}
	l_X = \max\{l^{\prime} \in \mathbb{N} \,:\, &\text{ for any two vertices } x \neq y \in X, \text{ there exist } \\
	& l^{\prime} \text{ edge-disjoint paths from $x$ to $y$}\}. 
\end{align*}
From this remark and a basic discussion for the passage time (see, e.g., \cite[Lemma~2.3]{FPP_history}), the following lemma holds. 
\begin{lem}
\label{lem:finite_moment}
Let $t_1, \ldots, t_{l_X}$ be independent copies of $t_e$ and let $k\geq 1$. Then $\mE\min\{t_1, \ldots, t_{l_X}\}^k < \infty$ holds if and only if $\mE T(0, x)^k < \infty$ holds  for all $x \in X$. 
\end{lem}

\subsection{Asymptotic speed of first passage time}
\label{subsec:asymptotic_speed}
Suppose the time distribution $\nu$ of each edge satisfies
\begin{equation}
\label{eq:assumption_for_timeconstant}
	\mE \min\{t_1, \ldots, t_{l_X}\} < \infty,  
\end{equation}
where $l_X$ is the edge connectivity of $X$ and $t_1, t_2, \ldots, t_{l_X}$ are independent copies of $t_e$. Similarly to the cubic model, we first prove the following proposition:  
\begin{prp}
\label{prp:time_constant}
For each $x \in \mathcal D$, the limit
\begin{align*}
	\mu(x) := \lim_{n \to \infty} \frac{T(0, nx)}{n}
\end{align*}
exists almost surely. Moreover, the function $\mu:\mathcal{D} \rightarrow \R$ depends on only $X$, $\nu$, and the period $\rho$.
\end{prp}
The following theorem~\cite[Theorem 1.10]{subadditive_ergodic} is essential for the proof of Proposition~\ref{prp:time_constant}. 
\begin{thm}[Subadditive ergodic theorem]
\rm
\label{thm:subadditive}
Suppose a sequence $(X_{m, n})_{0 \leq m < n}$ of random variables satisfies the following conditions: 
	\begin{itemize}
	\item $X_{0,n} \leq X_{0,m} + X_{m,n}$ for all $0 < m < n$; 
\item the joint distributions of the two sequences
\begin{equation*}
	(X_{m, m+k})_{k\geq 1} \text{ and } (X_{m+1, m+k+1})_{k\geq 1}
\end{equation*}
is the same for all $m \geq 0$; 
\item for each $k \geq 1$, the sequence $(X_{nk, (n+1)k})_{n\geq 0}$ is stationary and ergodic; and 
\item $\mE X_{0, 1} < \infty$ and $\mE X_{0, n} > -cn$ for some finite constant $c < \infty$. 
\end{itemize}
Then 
\begin{align*}
	\frac{X_{0, n}}{n} \longrightarrow \inf_n \frac{\mE X_{0, n} }{n} = \lim_{n\to \infty}\frac{\mE X_{0, n}}{n} 
\end{align*}
as $n \longrightarrow \infty$ almost surely and in $L_1$.
\end{thm}
We now turn to the proof of Proposition~\ref{prp:time_constant}. 
\begin{proof}[Proof of Proposition~\ref{prp:time_constant}]
	Take the minimum number $N \in \mathbb{N}$ with $Nx \in \Gamma$. We can easily see that the array $(T(mNx, nNx))_{0\leq m<n}$ of random variables satisfies the conditions of Theorem~\ref{thm:subadditive}. Note that the integrability of $X_{0,1}$ follows from the assumption \eqref{eq:assumption_for_timeconstant} and Lemma~\ref{lem:finite_moment}. Thus we see that the limit 
\begin{equation*}
	 \lim_{k \to \infty} \frac{T(0, kNx)}{k}
\end{equation*}
exists almost surely and is constant. We set
\begin{equation*}
	\mu(x):= \lim_{k \to \infty} \frac{T(0, kNx)}{kN}. 
\end{equation*}
Note that $\mu(x)$ depends on only $Nx$, the graph $X$ and the lattice group $\Gamma:= \rho(L)$. Take $j = 1, 2, \ldots, N-1$ arbitrarily. From the triangle inequality, we have
\begin{align*}
	|T(0, (kN+j)x) - T(0, kNx)| \leq T((kN+j)x, kNx)
\end{align*}
and thus, for any $\epsilon > 0$, 
\begin{align*}
	&\sum_{k=1}^{\infty}\mathbb{P}(|T(0, (kN+j)x) - T(0, kNx)| > \epsilon k) \\
\leq &\sum_{k=1}^{\infty} \mathbb{P}(T((kN+j)x, kNx) > \epsilon k)  \\
 =  & \sum_{k=1}^{\infty} \mathbb{P}(T(jx, 0) > \epsilon k )  < \infty. 
\end{align*}
Here, the first equality follows from Remark~\ref{rem:distribution_translation}, and the finiteness is due to the integrability of $T(jx, 0)$. Then it follows from the Borel--Cantelli lemma that
\begin{align*}
	\mathbb{P}(\limsup_{k \to \infty} \{ |T(0, (kN+j)x) - T(0, kNx)| > \epsilon k \}) = 0. 
\end{align*}
By taking the complementary event, we have
\begin{align*}
	\mathbb{P}\left( \bigcup_n \bigcap_{k \geq n} \{ |T(0, (kN+j)x) - T(0, kNx)| \leq \epsilon k \} \right) = 1,
\end{align*}
which implies the almost sure convergence
\begin{align*}
	\frac{1}{kN} |T(0, (kN+j)x) - T(0, kNx)| \longrightarrow 0
\end{align*}
as $k \longrightarrow \infty$. Therefore, we have 
\begin{align*}
	\frac{1}{kN+j}T(0, (kN+j)x) \longrightarrow \mu(x)
\end{align*}
almost surely, which implies $T(0, nx)/n \longrightarrow \mu(x)$ as $n \longrightarrow \infty$. 
\end{proof}

We summarize the basic properties of $\mu$. 
\begin{prp}
\rm
\label{prp:properties}
	The following hold:
\begin{enumerate}
\item $\mu(x+y) \leq \mu(x) + \mu(y)$ for any $x$, $y\in \mathcal D$;
\item $\mu(cx) = |c| \mu(x)$ for any $c \in \mathbb{Q}$ and $x\in \mathcal D$. 
\end{enumerate}
\end{prp}

\begin{proof}
Fix $x$, $y \in \mathcal D$. Let $N$ be the minimum number with $Nx$, $Ny \in \Gamma$. It follows from Remark~\ref{rem:distribution_translation} and Proposition~\ref{prp:time_constant} that 
\begin{equation}
\label{eq:translation}
	\frac{T(kNx, kNy)}{kN} \sim \frac{T(0, kN(y - x))}{kN} \longrightarrow_d \mu(y - x)
\end{equation}	
as $k \longrightarrow \infty$. Here, $\sim$ means the distributions are the same, and $\longrightarrow_d$ represents the convergence in distribution. By the definition of $T(\cdot, \cdot)$, we have the following triangle inequality
\begin{align}
\label{eq:triangle_inequality_of_T}
	\frac{T(0, kN(x+y))}{kN} - \frac{T(0, kNx)}{kN} \leq \frac{T(kNx, kN(x+y))}{kN}.  
\end{align}
The left hand side converges in distribution to $\mu(x+y)- \mu(x)$. From \eqref{eq:translation}, the right hand side converges to $\mu(y)$, and we obtain the first item. \par
Let $c \in \mathbb{Q}_{\geq 0}$. Then we have
\begin{align*}
	\mu(cx) = \lim_{n \to \infty} \frac{1}{n} T(0, ncx) = c \lim_{n \to \infty} \frac{1}{nc} T(0, ncx) = c \mu(x). 
\end{align*}
The symmetry of $T$ and  \eqref{eq:translation} implies that $\mu(x) = \mu(-x)$. Thus the second item  also holds for all $c \in \mathbb{Q}$. 
\end{proof}
Hereinafter, we think of $\mu$ as a function defined on the space $\R^d$ by continuous expansion. The following proposition states that the positivity of $\mu(x)$ depends on the probability $\nu(0)$ that the random time $t_e$ is equal to $0$. 
\begin{prp}
\rm
\label{prp:positivity}
The following hold: 
\begin{itemize}
	\item if $\nu(0) < p_c(X)$, then $\mu(x) > 0$ for all $x \in \R^d \setminus \{0\}$; and 
\item if $\nu(0) \geq p_c(X)$, then  $\mu(x) = 0$ for all $x \in \R^d$. 
\end{itemize}
\end{prp}
The proof of this proposition is similar to that of~\cite[Theorem~6.1]{Aspects}\footnote{Although \cite[Theorem~6.1]{Aspects} is stated by using another critical probability $p_T(X)$, the critical point at which the expected value of the cluster size becomes infinite, the uniqueness $p_c(X)=p_T(X)$ can be shown for arbitrary crystal lattice $X$ in the same way as $\mathbb{L}^d$ (see e.g. \cite{Grimmett}). }.  
 From Propositions~\ref{prp:properties} and \ref{prp:positivity}, the time constant $\mu$ is a norm on $\R^d$ whenever $\nu(0) < p_c(X)$. 
\par
We further assume that the time distribution $\nu$ of each edge satisfies
\begin{equation}
\label{eq:assumption_of_shape_theorem}
	\mE\min\{t_1, t_2, \ldots, t_{l_X} \}^d < \infty,
\end{equation}
which is an analogue of \eqref{eq:a_moment_condition}. Note that this condition is stronger than \eqref{eq:assumption_for_timeconstant}. The shape theorem follows from the following almost sure convergence: 
\begin{equation}
\label{eq:convergence_assumption}
	\lim_{x \in \mathcal D,\, \|x\|_1 \to \infty} \left(\frac{T(0, x)}{\|x\|_1} - \mu\left(\frac{x}{\|x\|_1}\right) \right) = 0,
\end{equation}
which states that the convergence $T(0, nx)/n \longrightarrow \mu(x)$ is uniform on the directions. The proof of \eqref{eq:convergence_assumption} is complicated, but it is essentially the same as that of \cite[Theorem 2.16]{FPP_history}. We provide this proof in the appendix of this paper.  
By combining with Proposition~\ref{prp:positivity}, we obtain the following, which is a generalization of the shape theorem.  

\begin{thm}
\rm
\label{thm:shape_theorem}
Let $(\Phi, \rho)$ be a periodic realization of a $d$-dimensional crystal lattice $X$. Suppose the time distribution $\nu$ satisfies \eqref{eq:assumption_of_shape_theorem}. Then the following hold: 
\begin{itemize}
	\item[(a)] If $\nu(0) < p_c(X)$, then for each $\epsilon > 0$, 
\begin{equation*}
	(1-\epsilon)\mB \subset \frac{B(t)}{t} \subset (1+\epsilon)\mB  \text{ for all large } t
\end{equation*}
holds almost surely, where the limit shape $\mB \subset \R^d$ is the unit ball
\begin{equation}
\label{eq:the_limit_shape}
	\mB = \{x \in \R^d \,:\, \mu(x) \leq 1\}. 
\end{equation}

\item[(b)] If $\nu(0) \geq p_c(X)$, then for all $R > 0$, 
\begin{equation*}
	\{x \in \R^d \,:\, \|x\|_1 \leq R \} \subset \frac{B(t)}{t} \text{ for all large } t
\end{equation*}
holds almost surely. 
\end{itemize}
\end{thm}

\subsection{Properties of the limit shape}
In this subsection, we summarize the basic relations between a periodic realization $(\Phi, \rho)$ and the limit shape, denoted by $\mathcal B_{\Phi}$. 
\begin{prp}
\rm
\label{prp:properties_of_limit}
The following hold:
\begin{enumerate}
\item The limit shape $\mB$ depends on only $X$, $\nu$, and the period $\rho$; that is, for two periodic realizations $(\Phi, \rho)$, $(\Phi^{\prime}, \rho^{\prime})$, we have $\mB_{\Phi} = \mB_{\Phi^{\prime}}$ whenever $\rho = \rho^{\prime}$. 
\item $\mathcal B_{\Phi} = \mathcal B_{\Phi + {\bf b}}$ for any ${\bf b} \in \R^d$, where $\Phi + {\bf b}$ is the periodic realization obtained by the map $x \mapsto \Phi(x) + {\bf b}$.  
\item $\, \mathcal B_{A\circ\Phi} = A\mathcal B_{\Phi}$ for any $A \in GL_d(\R)$ (note that $A\circ\Phi$ is also a periodic realization, whose period homomorphism is given by $A\circ\rho$). 
\end{enumerate}
\end{prp}
\begin{proof}
The first item follows from Proposition~\ref{prp:time_constant} and \eqref{eq:the_limit_shape}.\par
Fix ${\bf b} \in \R^d$ arbitrarily. Let $T^{\prime}(\cdot, \cdot)$ be the first passage time with respect to the realization $\Phi + {\bf b}$. Then we have
\begin{align*}
	T^{\prime}(0, x) = T(-{\bf b}, x - {\bf b}) \leq T(0, -{\bf b}) + T(0, x - {\bf b}),  
\end{align*}
which implies the inclusion
\begin{align*}
	B(t)+ {\bf b} \subset B^{\prime}(t+ T(0, -{\bf b})). 
\end{align*}
By dividing by $t$ and letting $t \longrightarrow \infty$, 
we obtain $\mathcal B_{\Phi} \subset \mathcal B_{\Phi + {\bf b}}$. By replacing ${\bf b}$ with $-{\bf b}$, we obtain the opposite inclusion. Thus, the proof of the second item is completed. \par
For the third item, denote by $\Gamma^{\prime}, T^{\prime}(\cdot, \cdot), \mu^{\prime}(\cdot)$ the characters with respect to $A\circ\Phi$. Then we have $\Gamma^{\prime} = A\Gamma$ and the relation
\begin{align*}
	T(0, x) = T^{\prime}(0, Ax)
\end{align*}
holds for any $x \in \Gamma$. This implies that $\mu(x) = \mu^{\prime}(Ax)$ holds for any $x \in \R^d$. From \eqref{eq:the_limit_shape}, the proof is completed. 
\end{proof}

\begin{ex}
\rm
In Fig.~\ref{fig:realizations}, since the realizations shown in the left and the center have the same period homomorphism $\rho$, the limit shapes obtained from the FPP model are the same, although the realizations are different. 
\end{ex}
\begin{rem}
\rm
\label{rem:not_essential}
	In Sect.~\ref{subsec:settings}, we assumed that the realization $\Phi$ is nondegenerate in order to formulate the FPP model on a crystal lattice. Proposition~\ref{prp:properties_of_limit} implies that this assumption is not essential for the limit shapes. Indeed, when we consider the FPP model on a lattice with a degenerate realization $(\Phi, \rho)$, we can obtain a nondegenerate one $(\Phi^{\prime}, \rho^{\prime})$ with $\rho = \rho^{\prime}$ and obtain the limit shape $\mB_{\Phi}$.
\end{rem}

Proposition~\ref{prp:properties_of_limit} gives the symmetric property of the limit shape. Let $\text{Sym}(\Phi(X))$ be the symmetric group of the image $\Phi(X)$, that is, 
\begin{align*}
	\text{Sym}(\Phi(X)) = \{g \in M(d) \,:\, g\Phi(X) = \Phi(X)\}, 
\end{align*}
where $M(d)$ is the group of congruent transformations of $\R^d$. We write $M(d)$ as the semi-product $M(d) = \R^d \rtimes O(d)$ of the translation $\R^d$ and the rotation $O(d)$.  Let $p: \text{Sym}(\Phi(X)) \rightarrow O(d)$ be the group homomorphism defined by
\begin{align*}
	({\bf b}, A) \mapsto A. 
\end{align*}
Then we obtain the following. 
\begin{prp}
\rm
	For any $A \in \text{Im}(p)$, $A\mathcal B_{\Phi} = \mathcal B_{\Phi}$. In other words, the limit shape $\mathcal B_{\Phi}$ has the symmetry given by $\text{Im} (p)$. 
\end{prp}
\begin{proof}
	For any $({\bf b}, A) \in \text{Sym}(\Phi(X))$, we obtain
\begin{align*}
	A\circ\Phi(X) + {\bf b} = \Phi(X), 
\end{align*}
which implies
\begin{align*}
	\mathcal B_{\Phi} = \mathcal B_{A\circ\Phi + {\bf b}} = \mathcal B_{A\circ\Phi} = A\mathcal B_{\Phi}. 
\end{align*}
Here we use the second and third items of Proposition~\ref{prp:properties_of_limit} for the second and third equalities, respectively. 
\end{proof}

\begin{ex}
\rm
As we can see in Fig.~\ref{fig:realizations}, the honeycomb lattice on the left has rotational symmetry, which implies the same symmetry of the limit shape. Note that the limit shape obtained from the lattice in the center is the same as that obtained from the one on the left. Thus, the lattice in the center also has rotational symmetry.   
\end{ex}

\section{Monotonicity of limit shapes}
\label{sec:covering_monotonicity}
\subsection{Setting and examples}
This section is devoted to the proof of Theorem~\ref{thm:covering_monotonicity}. First in this subsection, we give the setting and some examples. \par
Let $X=(V, E)$ be a crystal lattice with $d:=\dim X\geq 2$ and $\Phi:X \rightarrow \R^d$ be a periodic realization with a period homomorphism $\rho:L \rightarrow \R^d$. We consider the rational projection $P: \R^d \rightarrow \R^{d_1}$ onto some $d_1$-dimensional subspace $\R^{d_1}$. As we reviewed in Sect \ref{subsec:crystal_lattices}, the periodic realization $\Phi_1: X_1=(V_1, E_1)\rightarrow \R^{d_1}$ of the $d_1$-dimensional crystal lattice $X_1=X/\text{Ker}(P\circ \rho)$ is induced and the commutative diagram \eqref{diagram:projection} holds. For this $X$ and $X_1$, fix a time distribution $\nu$ such that the moment condition \eqref{eq:assumption_of_shape_theorem} holds for $X$ and $X_1$. Let the configuration space $(\Omega, \mathcal F, \mP)$ be the product space indexed by the disjoint union $E \sqcup E_1$, that is, 
\begin{equation}
\label{eq:probability_space}
	\Omega = [0, \infty)^{E \sqcup E_1}, \ \mP = \nu^{\otimes E\sqcup E_1} 
\end{equation}
and $\mathcal F$ is the Borel $\sigma$-algebra.\par
We fix a vertex $0 \in X$ as the origin and suppose $\Phi(0) = 0 \in \R^d$. We write $\omega(0) \in X_1$, which satisfies $\Phi_1(\omega(0)) = 0 \in \R^{d_1}$, by $0$ for short. We use the same notations as in the previous sections with respect to this process for $X$. For the covered graph $X_1$, we represent them with subscripts, such as $T_1$, $B_1(t)$, $\mu_1$, and $\mB_1$. If the case~(b) of Theorem~\ref{thm:shape_theorem} holds, then $\mB$ (resp. $\mB_1$) is the whole space $\R^d$ (resp. $\R^{d_1}$). \par
Before we turn to the proof, we introduce examples of the application of Theorem~\ref{thm:covering_monotonicity} and remark on several works related to it.  

\begin{ex}
\rm
By the orthogonal projection $P: \R^3 \rightarrow \R^2 \simeq \{(x, y, z) \in \R^3\,:\,x+y+z=0\}$, the cubic lattice $\mathbb{L}^3$ is projected onto the triangular lattice $\mathbb{T}$ realized in $\{(x, y, z) \in \R^3 \,:\, x+y+z=0\}$. The fact that $\mathbb{T}$ is a quotient graph of $\mathbb{L}^3$ is also shown by considering the action $\Z \curvearrowright \mathbb{L}^3$ given by the translation by the vector $(-1, -1, -1)$. \par
Theorem \ref{thm:covering_monotonicity} implies that the projection of the limit shape $\mB$ of $\mathbb{L}^3$ to the the plane $\{(x, y, z)\,:\,x+y+z=0\}$ is bounded below by the limit shape $\mB_1$ of $\mathbb{T}$.  
\end{ex}

\begin{ex}
\rm
Let $X$ be the cubic lattice $\bbL^2$. By the orthogonal projection $P: \R^2 \rightarrow W$ onto the subspace $W:= \{(x_1, x_2)\in \R^2 \,:\, x_2 = x_1\}$, we obtain a periodic realization of the quotient graph $X_1$, defined as the $1$-dimensional line with parallel edges (Fig.~\ref{fig:line_with_parallel}). From the law of large numbers, we can easily see that
\begin{equation*}
	\mu_1((1, 1)) = 2 \mE \min\{t_1, t_2\}
\end{equation*}
for a suitable time distribution $\nu$. From Theorem~\ref{thm:covering_monotonicity}, we can see 
\begin{equation}
	\{x \in W\,:\, \mu_1(x)\leq R\} \subset P(\{y \in \R^2\,:\, \mu(
	y)\leq R\})
\end{equation}
for any $R>0$. By setting $R:=\mu_1((1, 1))$, we can see that there exists $y \in P^{-1}((1,1))$ satisfying
\begin{equation}
\label{eq:upper_estimate}
	\mu(y) \leq \mu_1((1, 1)). 
\end{equation}
From the symmetric property of $\mathbb{L}^2$, the symmetric point $y^{\prime}$ of $y$ with respect to the line $W$ also satisfies \eqref{eq:upper_estimate}. Thus, we obtain
\begin{align*}
	\mu((1, 1)) = \mu((y + y^{\prime})/2) \leq \frac{1}{2}(\mu(y) + \mu(y^{\prime})) \leq \mu_1((1, 1)), 
\end{align*}
and we have the upper estimate
\begin{equation*}
	\mu((1,1)) \leq 2 \mE \min\{t_1, t_2\}. 
\end{equation*}
\end{ex}

\begin{figure}[H]
\captionsetup{width=0.85\linewidth}
\centering
\includegraphics[width = 3cm]{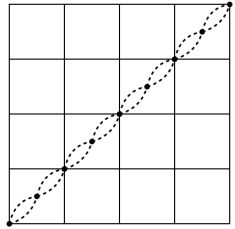}
 \caption{The cubic lattice $\bbL^2$ (black lines) and quotient graph $X_1$ (dots) realized in $W$. }
\label{fig:line_with_parallel}
\end{figure}

\begin{rem}
\rm
Theorem~\ref{thm:covering_monotonicity} implies the inequality
\begin{align}
\label{eq:monotonicity_for_bond_critical}
	p_c(X) \leq p_c(X_1),  
\end{align}
which was shown by~\cite{beyond_Z^d} in a more general setting. Indeed, the time distribution $\nu := p\delta_0 + (1-p)\delta_1$, where $\delta_a$ is the Dirac measure at $a \in \R$, satisfies the assumption of the shape theorem for any parameter $p \in [0, 1]$, and we have
\begin{align*}
	p =\nu(0) < p_c(X) &\Longleftrightarrow \mB\text{ is compact} \\ &\Longrightarrow \mB_1 \text{ is compact}  \Longleftrightarrow p<  p_c(X_1), 
\end{align*} 
which implies \eqref{eq:monotonicity_for_bond_critical}. Here, the second arrow follows from $\mB_1 \subset P(\mB)$. 
\end{rem}

\subsection{Proof of monotonicity}
We prove Theorem~\ref{thm:covering_monotonicity} in this subsection by giving some lemmas. The first lemma is a generalization of~\cite[Proposition 1.14]{Aspects}, stating that the asymptotic speed ``from point to line'' is equal to that ``from point to point''. Here, for a point $x \in \R^d$ and an affine subspace $A \subset \R^d$, we denote by $T(x, A)$ the first passage time from $x$ to $A$:
\begin{equation}
\label{eq:point_to_surface}
	T(x, A) = \inf_{y \in A} T(x, y). 
\end{equation}
The lemma is stated as follows. 
\begin{lem}
\rm
\label{lem:point_to_Affine}
Let $\nu$ be a time distribution satisfying \eqref{eq:assumption_of_shape_theorem} and $A$ be an affine subspace of $\R^d$. Then there exists a point $x \in A$ such that
\begin{equation*}
\label{eq:point_to_Affine}
	\mu(x) = \lim_{n \to \infty} \frac{T(0, nA)}{n}
\end{equation*}
holds almost surely. 
\end{lem}
\begin{proof}
This is clear for the case $0 \in A$ or $\nu(0)\geq p_c(X)$ ($\Longrightarrow \mu \equiv 0$). Suppose the case $0 \notin A$ and $\nu(0) <  p_c(X)$. Let $r > 0$ satisfy the affine subspace $rA$ being tangent to the limit shape $\mB= \{x \in \R^d \,:\, \mu(x) \leq 1 \}$, that is, 
\begin{itemize}
	\item $\mu(x) \geq 1$ for any $x \in rA$, and
\item there exists $y \in rA$ such that $\mu(y) = 1$. 
\end{itemize}
Fix a point $y \in rA$ with $\mu(y) = 1$. From the definition of the first passage time, we can easily see that
\begin{equation*}
	\limsup_{n \to \infty} \frac{T(0, nrA)}{n} \leq \mu(y) = 1
\end{equation*}
holds almost surely. We prove that
\begin{equation}
\label{eq:1_liminf}
	1 \leq \liminf_{n \to \infty} \frac{T(0, nrA)}{n}
\end{equation}
holds almost surely. Suppose \eqref{eq:1_liminf} does not hold almost surely. Then there exists $\delta > 0$ with $\mP(\Xi) > 0$, where $\Xi$ is the event defined as 
 \begin{equation}
\label{eq:limsup}
	\Xi :=\left\{ \liminf_{n \to \infty} \frac{T(0, nrA)}{n} \leq 1 - 4\delta \right\}. 
\end{equation}
Consider a configuration in the event $\Xi$. By \eqref{eq:limsup}, we can take a subsequence $\{n_k\}_k$ such that
\begin{equation*}
	\frac{T(0, n_k rA)}{n_k} \leq 1 - 3\delta 
\end{equation*}
holds for all large $k$. From \eqref{eq:point_to_surface}, we can take a sequence $y_1, y_2, \ldots$ of points with $y_k \in n_k rA$ such that
\begin{equation*}
\frac{T(0, y_k)}{n_k} \leq  \frac{T(0, n_k rA)}{n_k} + \delta. 
\end{equation*} 
Thus, we have 
\begin{equation*}
	\frac{T(0, y_k)}{n_k} \leq 1 -2\delta, 
\end{equation*}
which is equivalent to 
\begin{equation*}
	y_k \in B(n_k(1- 2\delta)). 
\end{equation*}
Take $\epsilon > 0$ small enough to satisfy $(1+ \epsilon)(1- 2\delta) \leq (1 - \delta)$. Then, Theorem \ref{thm:shape_theorem} implies that 
\begin{equation*}
	\frac{y_k}{n_k} \in \frac{B(n_k(1- 2\delta))}{n_k(1 - 2\delta)}(1 - 2\delta) \subset (1 - 2\delta)(1 + \epsilon)\mB \subset (1- \delta)\mB 
\end{equation*}
holds for all large $k$ almost surely. This leads to $\mu(y_k/n_k) \leq 1-\delta$. Since $y_k/{n_k} \in rA$, this contradicts the assumption that $rA$ is tangent to $\mB$. \par
From the above discussion, we obtain that 
\begin{align*}
	\mu(y) = \lim_{n \to \infty} \frac{T(0, nrA)}{n} = r \lim_{n \to \infty} \frac{T(0, nA)}{n} 
\end{align*}
holds almost surely. Setting $y^{\prime} := y/r \in A$, we have
\begin{align*}
	\mu(y^{\prime}) =  \lim_{n \to \infty} \frac{T(0, nA)}{n}, 
\end{align*}
which completes the proof of Lemma~\ref{lem:point_to_Affine}.  
\end{proof}

In the next lemma, we compare two passage times $T_1$ and $T$. Note that the following lemma itself does not assume any lattice structures of $X$ and $X_1$. 
\begin{lem}
\rm
\label{lem:probability_lift}
 For any vertex $x_1 \in X_1$ and $t \geq 0$, the inequality 
\begin{equation}
\label{eq:probability_lift}
	\mP(T_1(0, x_1) \geq t) \geq \mP(T(0, \wt x_1) \geq t \text{ for any }\wt x_1 \in \omega^{-1}(x_1)) 
\end{equation}
holds.  
\end{lem}
Theorem~\ref{thm:covering_monotonicity} follows from Lemmas~\ref{lem:point_to_Affine} and \ref{lem:probability_lift}. 

\begin{proof}[Proof of Theorem~\ref{thm:covering_monotonicity}]
	Take $x_1 \in \mathcal D_1$ with $\mu_1(x_1)\leq 1$ arbitrarily, and fix $N \in \mathbb{N}$ with $Nx_1 \in \Gamma_1$. Note that $kNx_1 \in X_1$ for any $k= 1, 2, \ldots$, and it follows from \eqref{eq:probability_lift} that
\begin{align*}
\label{eq:Nx_and_line}
	\mP(T_1(0, kNx_1) \geq t) &\geq \mP(T(0, y) \geq t \text{ for any } y \in \omega^{-1}(kNx_1)). 
\end{align*}
The projective relation \eqref{diagram:projection} implies that the right hand side is bounded below by
\begin{align*}
	\mP(T(0, P^{-1}(kNx_1)) \geq t) 
\end{align*}
since any points $y \in \omega^{-1}(kNx_1)$ are in the subspace $P^{-1}(kNx_1)$. By integrating with respect to $t$ from $0$ to $\infty$, we obtain
\begin{equation}
\label{eq:mu1_mu_compare}
 \mE T_1(0, kNx_1) \geq \mE T(0, P^{-1}(kNx_1)).  
\end{equation}
From Lemma~\ref{lem:point_to_Affine}, we can find a point $y \in P^{-1}(x_1)$ such that 
\begin{equation}
\label{eq:find_y}
	 \mu(y) = \lim_{k \to \infty} \frac{T(0, kN P^{-1}(x_1))}{kN}. 
\end{equation}
The sequence $\left(\frac{T(0, kN P^{-1}(x_1))}{kN}\right)_{k\geq 1}$ of random variables is uniformly integrable. Indeed, for some vertex $x \in \omega^{-1}(Nx_1)$, we have 
\begin{equation}
	\frac{T(0, kN P^{-1}(x_1))}{kN} \leq \frac{T(0, kx)}{kN} \leq \frac{1}{kN}\sum_{i=0}^{k-1} T(ix, (i+1)x).   
\end{equation}
From the assumption \eqref{eq:assumption_of_shape_theorem}, the $d$th moment of the right hand side is bounded above by some constant, which does not depend on $k$. Thus by taking expectation of \eqref{eq:find_y}, we obtain
\begin{equation}
\label{eq:DCT}
	\mu(y) = \lim_{k \to \infty}\frac{\mE T(0, kN P^{-1}(x_1))}{kN}. 
\end{equation}
By combining \eqref{eq:DCT} with \eqref{eq:mu1_mu_compare}, we obtain 
\begin{equation}
	1 \geq \mu_1(x_1) = \lim_{k \to \infty}\frac{\mE T_1(0,kNx_1)}{kN} \geq \lim_{k \to \infty}\frac{\mE T(0, kN P^{-1}(x_1))}{kN} = \mu(y),
\end{equation}
which implies $y \in \mB$. Here, the first equality follows from Theorem~\ref{thm:subadditive}. \par
We now obtain $\mB_1 \cap \mathcal D_1 \subset P(\mB)$. Since the projection $P(\mB)$ of the limit shape is closed, the proof of Theorem~\ref{thm:covering_monotonicity} is completed.  
\end{proof}

We next prove Lemma~\ref{lem:probability_lift}. The key idea of the proof is derived from the FKG inequality, which was first introduced by~\cite{FKG}. Here we introduce the statement of the FKG inequality in the context of the FPP model on $X$ and $X_1$. A partial order $\leq$ on the configuration space $\Omega$ is defined as
\begin{equation*}
	{\bf t} \leq {\bf t}^{\prime} \overset{\text{def}}\Longleftrightarrow t_e \leq t_e^{\prime} \text{ for any } e \in E\sqcup E_1
\end{equation*}
for two configurations ${\bf t} = (t_e \,:\, e \in E \sqcup E_1)$, ${\bf t}^{\prime} = (t^{\prime}_e \,:\, e \in E \sqcup E_1 ) \in \Omega$. An event $A$ is called \emph{increasing} if ${\bf t}^{\prime} \in A$ whenever ${\bf t} \in A$ and ${\bf t} \leq {\bf t}^{\prime}$. The simplest form of the FKG inequality is the following. 
\begin{thm}(FKG inequality)
	Let $A$ and $B$ be two increasing events, then
\begin{equation}
\label{eq:FKG_inequality}
	\mP(A \cap B) \geq \mP(A)\mP(B). 
\end{equation}
\end{thm}
We remark that the right hand side of \eqref{eq:FKG_inequality} can be regarded as the probability $\mP(A^{\prime}\cap B^{\prime})$ of the intersection $A^{\prime}\cap B^{\prime}$, where $A^{\prime}, B^{\prime}$ are independent copies of $A, B$. Here, we can roughly expect that the probability $\mP(A \cap B)$ decreases as the correlation of $A$ and $B$ decreases. For two paths $\gamma_1, \gamma_2$ in $X_1$ and their liftings $\wt \gamma_1, \wt \gamma_2$, the ``correlation'' of $\wt \gamma_1, \wt \gamma_2$ is less than that of $\gamma_1, \gamma_2$. Thus, by regarding the event $\{T_1(0, x_1) \geq t\}$ as the intersection $\bigcap_{\gamma}\{T(\gamma) \geq t\}$ of the events $\{T(\gamma) \geq t\}$ for all paths from 0 to $x_1$, and comparing it with $\bigcap_{\gamma}\{T(\wt \gamma) \geq t\}$, we can give a proof of Lemma~\ref{lem:probability_lift}.  \par
For a rigorous proof of Lemma~\ref{lem:probability_lift}, we require one more lemma. Let $A$ and $B$ be two increasing events which depend only on the family $(t_1, t_2, \ldots, t_m)$ of i.i.d random  variables with the distribution $\nu$. Let $t_a$ and $t_b$ be independent copies of $t_m$. We define $A^{\prime}$ (resp. $B^{\prime}$) as the event which is obtained from $A$ (resp. $B$) by replacing $t_m$ with $t_a$ (resp. $t_b$). Then the following holds. 
\begin{lem}
\rm
\label{lem:corelation_lift} 
\begin{equation}
\label{eq:corelation_lift}
\mP(A \cap B) \geq \mP(A^{\prime} \cap B^{\prime}). 
\end{equation}
\end{lem}

\begin{proof}
Let $[m]:= \{1,2, \ldots, m\}$. We denote by $\mP_{\bullet}:= \nu^{\otimes \bullet}$ the product measure on $[0, \infty)^{\bullet}$. Since $A$ and $B$ do not depend on $t_a, t_b$, we can identify these events with the Borel subsets of $[0, \infty)^{[m]}$, and we have
\begin{align}
&\mP(A \cap B) \\
&= \int_{[0, \infty)^{[m]}} I_{A \cap B} d\mP_{[m]} \\
&= \int_{[0, \infty)^{[m-1]}}\int_{[0, \infty)^{\{m\}}} I_{A \cap B} d\mP_{\{m\}} d\mP_{[m-1]} \\
&= \int_{[0, \infty)^{[m-1]}}
   \mP_{\{m\}}\big(A_{(t_1, \ldots, t_{m-1})} \cap B_{(t_1, \ldots, t_{m-1})}\big)d\mP_{[m-1]}. 
\label{eq:Fubini_conditioning}
\end{align}
 Here, $I$ is the indicator function and we denote by $A_{(t_1, \ldots, t_{m-1})}$ the set of $t_m \in [0, \infty)^{\{m\}}$ with $I_A(t_1, \ldots, t_m) = 1$. We easily see
\begin{equation}
\label{eq:prob=probprob}
	\mP_{\{m\}}\big(A_{(t_1, \ldots, t_{m-1})} \cap B_{(t_1, \ldots, t_{m-1})}\big)
	\geq \mP_{\{m\}}\big(A_{(t_1, \ldots, t_{m-1})}\big)\mP_{\{m\}}\big(B_{(t_1, \ldots, t_{m-1})}\big). 
\end{equation}
Indeed, since both $A_{(t_1, \ldots, t_{m-1})}$ and $B_{(t_1, \ldots, t_{m-1})}$ are increasing subsets of the half line $[0, \infty)^{\{m\}}$, one of them is included in the other. Thus the left hand side of \eqref{eq:prob=probprob} is equal to one of the two  probabilities of the right hand side. \par
We identify the events $A^{\prime}$, $B^{\prime}$ with the Borel subsets of $[0, \infty)^{[m-1]\sqcup \{a, b\}}$. Define $A^{\prime}_{(t_1, \ldots, t_{m-1})}$, $B^{\prime}_{(t_1, \ldots, t_{m-1})}$ as the sets of $(t_a, t_b) \in [0, \infty)^{\{a, b\}}$ with $I_{A^{\prime}}(t_1, \ldots, t_{m-1}, t_a, t_b) = 1$, $I_{B^{\prime}}(t_1, \ldots, t_{m-1}, t_a, t_b) = 1$, respectively. Then the right hand side of \eqref{eq:prob=probprob} is equal to   
\begin{equation}
\label{eq:measure_product}
 \mP_{\{a, b\}}\big(A^{\prime}_{(t_1, \ldots, t_{m-1})}\big)\mP_{\{a, b\}} \big(B^{\prime}_{(t_1, \ldots, t_{m-1})}\big) = \mP_{\{a, b\}}\big(A^{\prime}_{(t_1, \ldots, t_{m-1})} \cap B^{\prime}_{(t_1, \ldots, t_{m-1})}\big).  
\end{equation}
By combining \eqref{eq:Fubini_conditioning} with \eqref{eq:prob=probprob} and \eqref{eq:measure_product}, we obtain
\begin{align}
\mP(A \cap B) 
&\geq \int_{[0, \infty)^{[m-1]}}
  \mP_{\{a, b\}}\big(A^{\prime}_{(t_1, \ldots, t_{m-1})} \cap B^{\prime}_{(t_1, \ldots, t_{m-1})}\big) d\mP_{[m-1]} \\
&= \int_{[0, \infty)^{[m-1]}}\int_{[0, \infty)^{\{a, b\}}}
  I_{A^{\prime}\cap B^{\prime}}d\mP_{\{a, b\}} d\mP_{[m-1]} \\
&=\mP(A^{\prime} \cap B^{\prime}), 
\end{align}
which completes the proof of Lemma \ref{lem:corelation_lift}. 
\end{proof}
We now turn to the proof of Lemma~\ref{lem:probability_lift}. 
\begin{proof}[Proof of Lemma~\ref{lem:probability_lift}]
Let $\gamma_1, \ldots \gamma_n$ be arbitrary self-avoiding paths in $X_1$ from the origin $0$ to $x_1$, and let $\wt \gamma_1, \ldots, \wt \gamma_n$ in $X$ be their liftings. We first show 
\begin{equation}
\label{eq:n_paths_lifting}
	\mP\left(\bigcap_{i=1}^n\{T_1(\gamma_i)\geq t \}\right) \geq \mP\left(\bigcap_{i=1}^n\{T(\wt \gamma_i)\geq t \}\right)
\end{equation}
for any $t \geq 0$. We set $\gamma_i = (e_{i, 1}, \ldots, e_{i, r_i})$ and $\wt \gamma_i = (\wt e_{i, 1}, \ldots, \wt e_{i, r_i})$ for $i = 1,2, \ldots, n$. Note that $\wt e_{i, j} \in E$ is mapped to $e_{i, j} \in E_1$ by the covering map $\omega:X \rightarrow X_1$. \par
Let $\mathcal I:= \{(i, j)\,:\, i=1,2,\ldots, n \text{ and } j=1,2,\ldots, r_i\}$ be the index set. For a partition $\mathcal S=\{S_1, \ldots, S_m\}$ of $\mathcal I$, we denote by $\pi_{\mathcal S}: \mathcal I \rightarrow \mathcal S$ the canonical map, which is defined by $\pi_{\mathcal S}(i, j) = S_k$ when $(i, j) \in S_k$. We define the probability $\mP^{(\mathcal S)} \in [0, 1]$ by
\begin{equation}
\label{eq:def_of_prob}
	\mP^{(\mathcal S)} := \mP\left(\bigcap_{i=1}^{n} \left\{\sum_{j=1}^{r_i} t_{\pi_{\mathcal S}(i,j)}\geq t\right\}\right), 
\end{equation}
where $(t_S\,:\, S \in \mathcal S)$ is the $\mathcal S$-indexed family of i.i.d random variables with the distribution $\nu$. 
\par
We set the partition $\mathcal S$ (resp. $\wt{\mathcal S}$) of $\mathcal I$ by the equivalence relation
\begin{align}
	&(i, j) \sim (i^{\prime}, j^{\prime}) \overset{\text{def}}\Longleftrightarrow 
	e_{i,j} = e_{i^{\prime}, j^{\prime}} \\(\text{resp. } &(i, j) \sim (i^{\prime}, j^{\prime}) \overset{\text{def}}\Longleftrightarrow 
	\wt e_{i,j} = \wt e_{i^{\prime}, j^{\prime}}). 
\end{align}
The inequality \eqref{eq:n_paths_lifting} can be rewritten as 
\begin{equation}
\label{eq:partition_probability}
	\mP^{(\mathcal S)} \geq \mP^{(\wt{\mathcal S})}. 
\end{equation}
Since $\wt e_{i,j} = \wt e_{i^{\prime}, j^{\prime}}$ implies $e_{i,j} = e_{i^{\prime}, j^{\prime}}$, the partition $\wt {\mathcal S}$ is finer than $\mathcal S$. We show that the further division of the partition $\mathcal S$ decreases the probability \eqref{eq:def_of_prob}. \par
Let $\mathcal S^{\prime}$ be a partition obtained from $\mathcal S := \{S_1, \ldots, S_m\}$ by splitting some element, say $S_m \in \mathcal S$, into two nonempty subsets $S_a$, $S_b$. For the $\mathcal S$-indexed family $(t_S\,:\, S \in \mathcal S)$ of i.i.d random variables, we take two independent copies $t_{S_a}$, $t_{S_b}$ of $t_{S_m}$ and obtain the $\mathcal S^{\prime}$-indexed family $(t_S\,:\, S \in \mathcal S^{\prime})$. Let $p: \mathcal I \rightarrow \{1,2,\ldots, n\}$ be the projection defined by $(i, j)\mapsto i$. We write the event in \eqref{eq:def_of_prob} as the intersection $A \cap B$ of the two events $A$, $B$, where 
\begin{equation}
A = \bigcap_{i \in p(S_a)} \left\{\sum_{j=1}^{r_i} t_{\pi_{\mathcal S}(i,j)}\geq t\right\}
\text{ and }
 B = \bigcap_{i \notin p(S_a)} \left\{\sum_{j=1}^{r_i} t_{\pi_{\mathcal S}(i,j)}\geq t\right\}.  
\end{equation}
We also set
\begin{equation}
A^{\prime} = \bigcap_{i \in p(S_a)} \left\{\sum_{j=1}^{r_i} t_{\pi_{{\mathcal S^{\prime}}}(i,j)}\geq t\right\}
\text{ and }
B^{\prime} = \bigcap_{i \notin p(S_a)} \left\{\sum_{j=1}^{r_i} t_{\pi_{{\mathcal S^{\prime}}}(i,j)}\geq t\right\}.  
\end{equation}
Here, we note that the event $A^{\prime}$ does not depends on $t_{S_b}$. Indeed, the assumption that each $\gamma_i$ is self-avoiding implies that the restriction $p_{\restriction_{S_k}}$ of the map $p$ to some element $S_k \in \mathcal S$ is injective. Since $S_a$ and $S_b$ are disjoint subsets of the same element $S_m \in \mathcal S$, we have $p(S_a)\cap p(S_b) = \emptyset$. 
\par
Therefore, the event $A^{\prime}$ (resp. $B^{\prime}$) can also be obtained from $A$ (resp. $B$) by replacing $t_{S_m}$ with $t_{S_a}$ (resp. $t_{S_b}$). From Lemma~\ref{lem:corelation_lift}, we obtain
\begin{align}
\label{eq:first_step}
	 \mP^{(\mathcal S)} \geq \mP^{(\mathcal S^{\prime})}.  
\end{align}
We can take a finite sequence $\mathcal S = \mathcal S^{(0)}, \mathcal S^{(1)}, \ldots, \mathcal S^{(K)} = \wt{\mathcal S}$ of partitions of $\mathcal I$ such that $\mathcal S^{(k+1)}$ is obtained from $\mathcal S^{(k)}$ by splitting some element of $\mathcal S^{(k)}$ into two nonempty sets. By replacing $\mathcal S$, $\mathcal S^{\prime}$ with $\mathcal S^{(k)}$, $\mathcal S^{(k+1)}$ and iterating the above discussion for $k=0, 1, \ldots, K-1$, we obtain \eqref{eq:partition_probability}. This completes the proof of \eqref{eq:n_paths_lifting}. \par
We next show \eqref{eq:probability_lift} from \eqref{eq:n_paths_lifting}. Let $\Lambda_R$ be the ball  with radius $R$: 
\begin{equation*}
	\Lambda_R = \{y_1 \in X_1 \,:\, d_{X_1}(0, y_1) \leq R \}, 
\end{equation*}
where $d_{X_1}$ is the graph metric. Letting $R$ be sufficiently large that $\Lambda_R$ includes $x_1$, we set the restricted first passage time
\begin{equation*}
	T_1^{R}(0, x_1) := \inf \{ T_1(\gamma) \,:\, \gamma \mbox{ is a path in $\Lambda_R$ from $0$ to $x_1$} \}. 
\end{equation*}
Let $\{\gamma_1, \gamma_2, \ldots, \gamma_n\}$ be the finite set of all self-avoiding paths in $\Lambda_R$ that go from $0$ to $x_1$. Then we have
\begin{equation*}
	T_1^{R}(0, x_1) := \min_{i=1, \ldots, n} T_1(\gamma_i),\end{equation*}
and it follows from \eqref{eq:n_paths_lifting} that 
\begin{align*}
	\mP(T_1^{R}(0, x_1) \geq t) 
= \mP\left(\bigcap_{i=1}^n \{T_1(\gamma_i) \geq t\} \right) \geq \mP\left(\bigcap_{i=1}^n \{T(\wt  \gamma_i) \geq t\} \right), 
\end{align*}
where each $\wt \gamma_i$ is the lifting of $\gamma_i$. Since the terminus of each path $\wt \gamma_i$ is in $\omega^{-1}(x_1)$, the last expression is bounded below by
\begin{equation*}
	\mP(T(0, \wt x_1) \geq t \text{ for any }\wt x_1 \in \omega^{-1}(x_1)). 
\end{equation*}
Letting $R \longrightarrow \infty$ completes the proof of Lemma~\ref{lem:probability_lift}. 
\end{proof}

\section{Conclusion}
We introduced a general version of the FPP model defined on crystal lattices, and showed the covering monotonicity of the limit shape. Then, in light of the result of strict monotonicity in~\cite{Strict_monotonicity}, the following problem is of interest for further study.  
\begin{itemize}
	\item Is the inclusion $\mB_1 \subset P(\mB)$ of Theorem~\ref{thm:covering_monotonicity} strict?
\end{itemize}
Another main interest is derived from the formulation of crystal lattices. As mentioned in Sect.~\ref{sec:introduction}, one of the main studies of crystal lattices comes from the concept of \emph{standard realization}, periodic realization with the maximal symmetry. The cubic lattice that we often consider in percolation theory is one example. The paper~\cite{Simulation} provides a simulation study of the FPP model on a $2$-dimensional cubic lattice, and observes that the larger the variability of the time distribution is, the closer the limit shape becomes to the circle. We thus pose the following question: 
\begin{itemize}
	\item For the standard realization $\Phi:X \rightarrow \R^d$ of any  crystal lattice $X$, does the limit shape become closer to the sphere as the variance of time distribution increases? 
\end{itemize}

\section*{Acknowledgement}
The author would like to thank Masato Takei and Shuta Nakajima for their valuable suggestions and comments. This work was supported by JSPS KAKENHI Grant Number 19J20795.

\appendix
\section{Appendix}

\subsection{Tree-lifting property}
\label{subsec:tree_lifting}
Let $\omega: X \rightarrow X_0$ be a regular covering graph over a finite graph $X_0$.  
Let $\mT_0$ be a spanning tree of $X_0$. Fix two vertices $x_0 \in X_0$ and $x \in X$ with $\omega(x) = x_0$. Then there exists a unique subtree $\mT \subset X$, which we call a  \emph{lifting} of $\mT_0$, satisfying $x \in \mT$ and the restriction $\omega_{\restriction_{\mT}}: \mT \rightarrow \mT_0$ of $\omega$ is an isomorphism. Indeed, we can construct $\mT$ as
\begin{equation*}
	\mT := \bigcup_{i} \gamma_i,
\end{equation*}
where each path $\gamma_i$ is the lifting of the unique path in $\mT_0$ from $x_0$ to each leaf $y_0^i$. The uniqueness follows from the unique path-lifting property. For this tree $\mT$ and $\sigma \in G(\omega)$, the translation $\sigma(\mT)$, denoted by $\mT_{\sigma}$, is the unique lifting of $\mT_0$ containing $\sigma x$. The following proposition states that the vertex set of a covering graph can be represented as the array of a spanning tree of $X_0$. 
\begin{prp}
\rm
\label{prp:disjoint_trees}
Let $\omega: X = (V, E) \rightarrow X_0$ be a regular covering over a finite graph $X_0$. Fix a spanning tree $\mT_0 \subset X_0$ and its lifting $\mT \subset X$. Then the following holds:
\begin{equation*}
	V = \bigsqcup_{\sigma \in G(\omega)} \mT_{\sigma}, 
\end{equation*}
where $G(\omega)$ is the covering transformation group of $\omega$. 
\end{prp}
\begin{proof}
Take $y \in V$ arbitrarily. For $y_0:= \omega(y)$, we can find $y^{\prime} \in \mT$ with $\omega(y^{\prime}) = y_0$ since $\omega_{\restriction_{\mT}}$ is a bijection.  From the regularity of $\omega$, there exists $\sigma \in G(\omega)$ such that $y = \sigma(y^{\prime})$, which implies $y \in \mT_{\sigma}$. Thus $V \subset \cup_{\sigma \in G(\omega)} \mT_{\sigma}$. The disjointness of the right hand side follows from the uniqueness of the lifting tree. 
\end{proof}

\subsection{Proof of the shape theorem}
In this subsection, we give the proof of \eqref{eq:convergence_assumption}. If $d=1$, then \eqref{eq:convergence_assumption} directly follows from Proposition~\ref{prp:time_constant}. For $d \geq 2$, we only need to check the following stochastic estimate, which is an analogue of \cite[Lemma 2.20]{FPP_history}, for the passage time $T(x, y)$. 
\begin{lem}
\rm
\label{lem:kappa}
Suppose $d \geq 2$ and the time distribution $\nu$ satisfies \eqref{eq:assumption_of_shape_theorem}. Then there exists a constant $\kappa < \infty$ such that
\begin{align*}
	\mathbb{P} \left(\sup_{x \in X, x \neq 0} \frac{T(0, x)}{\|x\|_1} < \kappa \right) > 0.
\end{align*}
\end{lem}
For the remaining discussion for the proof of \eqref{eq:convergence_assumption}, we refer to the proof of \cite[Theorem~2.16]{FPP_history}. Although the outline of the proof of Lemma~\ref{lem:kappa} is similar to that of \cite[Lemma 2.20]{FPP_history}, we need to consider the array of trees in $X$ as an analogy for the graph structure of the cubic lattice.
\begin{proof}[proof of Lemma A.1]
We first check that Lemma \ref{lem:kappa} follows from the following estimate 
\begin{equation}
\label{eq:sum_all_vertices}
	\sum_{x \in X} \mP(T(0, x) \geq C\|x\|_1) < \infty
\end{equation}
for some constant $C > 0$. Let $\{x_1, x_2, \ldots \}$ be an ordering of the vertex set $X\setminus \{0\}$. From \eqref{eq:sum_all_vertices}, we have
\begin{equation*}
	\mP\left(\bigcup_{n \geq N} \{T(0, x_n) \geq C\|x_n\|_1\}\right) \leq 
	\sum_{n\geq N} \mP(T(0, x_n) \geq C\|x_n\|_1) < 1/3
\end{equation*}
for large $N$. By taking the complement, we obtain
\begin{equation*}
	\mP\left(\sup_{n \geq N}\frac{T(0, x_n)}{\|x_n\|_1} \leq C\right) > 2/3, 
\end{equation*}
and we can find $\kappa^{\prime} > 0$ satisfying 
\begin{equation*}
	\mP\left(\max_{n =1, \ldots, N-1}\frac{T(0, x_n)}{\|x_n\|_1} \leq \kappa^{\prime}\right) > 2/3. 
\end{equation*}
Letting $\kappa := \max\{\kappa^{\prime}, C\}$ implies Lemma~\ref{lem:kappa}. \par
We now turn to the proof of \eqref{eq:sum_all_vertices}. Let us recall the discussion of Sect.~\ref{subsec:tree_lifting}. From Proposition~\ref{prp:disjoint_trees}, the vertex set $V$ of $X$ can be divided into the liftings of a spanning tree of the base graph $X_0$:
\begin{align*}
	V = \bigsqcup_{z \in \Z^d} \mT_z. 
\end{align*}
Here, we identify the free abelian group $L$ with $\Z^d$ by taking some $\Z$-basis of $L$. We denote by $0_{\Z^d}$ the identity element of $\Z^d$ in order to distinguish from the origin $0 \in X$. We denote by $|\sigma|$ the $L_1$-norm of $\sigma \in \Z^d$, and let $\sigma \sim \tau$ mean $|\sigma - \tau| = 1$. Let $l_X$ be the edge connectivity of $X$. For a number $R \in \mathbb{N}_{>0}$ and $\sigma \in \Z^d$, we define a box $\Lambda(\sigma) \subset \Z^d$ as
\begin{equation*}
	\Lambda(\sigma) := 2R\sigma + (-R, R]^d \cap \Z^d. 
\end{equation*}
We fix $R$ sufficiently large so that the following condition holds (Fig.~\ref{fig:box_covering}): for any action $\sigma \sim 0_{\Z^d}$, there exist $l_X$ edge-disjoint paths $\gamma^{(\sigma)}_1,\ldots,\gamma^{(\sigma)}_{l_X}$ satisfying that 
\begin{itemize}
	\item all vertices of the paths are in $\bigsqcup_{z \in \Lambda\big(0_{\Z^d}\big) \cup \Lambda(\sigma)} {\mathcal T}_z$; and
	\item each path connects $0 \in X$ and $(2R\sigma)0 \in X$. 
\end{itemize}
The existence of this $R$ follows from the periodic structure of the graph $X$.  
\begin{figure}[H]
\captionsetup{width=0.85\linewidth}
\centering
\includegraphics[width = 5cm]{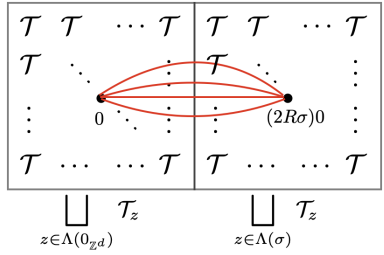}
 \caption{The $l_X$ paths $\gamma^{(\sigma)}_1,\ldots,\gamma^{(\sigma)}_{l_X}$ from $0$ to $(2R\sigma)0$ (red).}
\label{fig:box_covering}
\end{figure}
For two actions $\sigma, \tau$ with $\sigma \sim \tau$, let $T(\sigma, \tau)$ be the minimum passage time of the $l_X$ paths $(2R\sigma)\gamma^{(\sigma^{-1}\tau)}_1,\ldots, (2R\sigma)\gamma^{(\sigma^{-1}\tau)}_{l_X}$, that connect $(2R\sigma)0$ and $(2R\tau)0$. For a self-avoiding path $\pi = (0_{\Z^d}=\sigma_0, \sigma_1, \ldots, \sigma_m=\sigma)$ of length $m$ in $\Z^d$, we have
\begin{equation}
\label{eq:subadditive_sum}
	T(0, (2R\sigma)0) \leq \sum_{i=0}^{m-1} T(\sigma_{i}, \sigma_{i+1}). 
\end{equation}
We let $T_i := T(\sigma_{i}, \sigma_{i+1})$ and denote by $T(\pi)$ the right hand side of \eqref{eq:subadditive_sum}. By Lemma~\ref{lem:finite_moment} and the assumption \eqref{eq:assumption_of_shape_theorem}, each $T_i$ has a finite $d$th moment. Since we assume $d\geq 2$, the variance $\V(T_i)$ of $T_i$ is finite. We set
\begin{align}
	\V_{\max} := \max_{0_{\Z^d} \sim \tau}  \V(T(0_{\Z^d}, \tau))
\end{align}
and
\begin{align}
	\mE_{\max} := \max_{0_{\Z^d} \sim \tau} \mE T(0_{\Z^d}, \tau). 
\end{align}
Since $T_i$ and $T_j$ are independent whenever $|i-j|>1$, the variance $\V(T(\pi)) $ of $T(\pi)$ is equal to
\begin{align*}
	\V(T(\pi)) = \sum_{i=0}^{m-1} \mE (T_i - \mE T_i)^2 + 2\sum_{i=0}^{m-2} \mE (T_i - \mE T_i)(T_{i+1} - \mE T_{i+1}).  
\end{align*}
This variance is bounded above by $3m\V_{\max}$. Indeed, it follows from the Cauchy-Schwartz inequality that
\begin{align*}
	\mE |(T_i - \mE T_i)(T_{i+1} - \mE T_{i+1})| \leq 
\V(T_i)^{1/2}\V(T_{i+1})^{1/2} \leq \V_{\max}.   
\end{align*}
Chebyshev's inequality implies that
\begin{align}
\label{eq:variance_upperbound}
	\mP\left(T(\pi) \geq \sum_{i=0}^{m-1}(\mE_{\max} + 1)\right)
\leq &\mP\left(\sum_{i=0}^{m-1} (T_i - \mE T_i ) \geq m\right)\notag\\
\leq &\frac{1}{m^2}\V(T(\pi))\\
\leq &3\V_{\max}/m. 
\end{align}
To prove \eqref{eq:sum_all_vertices}, we need to improve this estimate. For each $\sigma \in \Z^d\setminus \{0_{\Z^d}\}$, we can take $2d$ paths $\pi_1(\sigma), \pi_2(\sigma), \ldots, \pi_{2d}(\sigma)$ in $\Z^d$ that connect $0_{\Z^d}$ and $\sigma$ and satisfy the following conditions:
\begin{itemize}
	\item they are vertex-disjoint except for $0_{\Z^d}$ and $\sigma$; and 
\item the length of each path is less than $|\sigma| + K_d$.  
\end{itemize}
Here, the constant $K_d$, depending on $d$, is the cost for making a detour in order that these paths do not overlap. Fix $\sigma \in \Z^d \setminus \{0_{\Z^d}\}$ and let $\pi_1, \ldots, \pi_{2d}$ be paths satisfying the above conditions. We consider the separation $\pi_j = \pi_j^1+ \pi_j^2 +  \pi_j^3$ of each path $\pi_j$, where $\pi_j^1$ (resp. $\pi_j^3$) is the first (resp. last) step of $\pi_j$. Let
\begin{align*}
	U := \max_{j=1,2,\ldots,2d} T(\pi_j^1) \text{ and }  U^{\prime} := \max_{j=1,2,\ldots,2d} T(\pi_j^3). 
\end{align*} 
Then we obtain $T(0, (2R\sigma)0) \leq U + \min_{j=1,2,\ldots, 2d} T(\pi_j^2) + U^{\prime}$, and 
\begin{align*}
 &\mP\big(T(0, (2R\sigma)0) \geq (\mE_{\max} + 1)(|\sigma|+K_d) + 2|\sigma|\big) \\
\leq &\mP(U \geq |\sigma|) + \mP(U^{\prime} \geq |\sigma|) + \mP\big(\min_{j=1,2,\ldots, 2d} T(\pi_j^2) \geq (\mE_{\max} + 1)(|\sigma|+K_d)\big). 
\end{align*}
The first and second terms of the right hand side are bounded above by
\begin{align*}
	2d \max_{0_{\Z^d} \sim \tau}\mP(T(0_{\Z^d}, \tau) \geq |\sigma|).  
\end{align*}
This is summable in $\sigma \in \Z^d$ since $T(0_{\Z^d}, \tau)$ has a finite $d$th moment. The independence of $T(\pi_j^2) \, (j = 1,2, \ldots, 2d)$ implies that the last term is equal to 
\begin{align*}
	\prod_{j=1}^{2d} \mP\big(T(\pi_j^2) \geq (\mE_{\max} + 1)(|\sigma|+K_d)\big).  
\end{align*}
Since the length of each path $\pi_j^2$ is less than $|\sigma|+K_d$, it follows from \eqref{eq:variance_upperbound} that this is bounded above by
\begin{equation*}
	3^{2d}\left(\frac{\V_{\max}}{|\sigma|+K_d}\right)^{2d},  
\end{equation*} 
which is summable in $\sigma \in \Z^d$. Now we have proved that 
\begin{equation}
\label{eq:Rsigma_finite}
	\sum_{\sigma \in \Z^d} \mP(T(0, (2R\sigma)0) \geq C_1 |\sigma|) < \infty, 
\end{equation}
where $C_1$ is a constant satisfying $C_1 |\sigma| \geq (\mE_{\max} + 1)(|\sigma|+K_d) + 2 |\sigma|$ for all $\sigma \neq 0_{\Z^d}$. \par
Finally, we consider the passage time for all vertices that do not necessarily coincide with $(2R\sigma) 0$.  Since the number of vertices in each box $\bigsqcup_{z \in \Lambda(\sigma)} \mT_z$ is finite, it follows from Lemma~\ref{lem:finite_moment}  that the random variable
\begin{align*}
	S(\sigma):= \max \left\{T((2R\sigma)0, x)\,:\, x \in \bigsqcup_{z \in \Lambda(\sigma)} \mT_z \right\}
\end{align*} 
has a finite $d$th moment. For any vertex $x \in X$, the first passage time $T(0, x)$ is bounded above by
\begin{equation*}
	T(0, (2R\sigma)0) + S(\sigma),
\end{equation*}
where $\sigma = \sigma_x \in \Z^d$ is the action satisfying $x \in \bigsqcup_{z \in \Lambda(\sigma)}\mathcal T_z$. We take a constant $C_2$ such that $C_2\|x\|_1 \geq |\sigma_x|$ holds for any $x \in X$ and set $C := 2C_1C_2$. Then we obtain
\begin{align*}
	&\sum_{x \in X} \mP(T(0, x)\geq C\|x\|_1) \\
\leq & \sum_{\sigma \in \Z^d} \sum_{x \in \bigsqcup_{z \in \Lambda(\sigma)}\mT_{z}} \left[ \mP(T(0, (2R\sigma)0)\geq C_1|\sigma|) +\mP\left(S(\sigma) \geq C_1|\sigma|\right) \right] \\
= & \left|\bigsqcup_{z \in \Lambda(\sigma)}\mT_{z}\right| \sum_{\sigma \in \Z^d} \left[ \mP(T(0, (2R\sigma)0)\geq C_1|\sigma|) + \mP\left(S(\sigma) \geq C_1|\sigma|\right) \right] \\
= & \left|\bigsqcup_{z \in \Lambda(\sigma)} \mT_z \right|\left[\sum_{\sigma \in \Z^d} \mP(T(0, (2R\sigma)0)\geq C_1|\sigma|) + \sum_{\sigma \in \Z^d} \mP\left(S(0_{\Z^d}) \geq C_1|\sigma|\right)\right],
\end{align*}
where $\left|\bigsqcup_{z \in \Lambda(\sigma)}\mT_{z}\right|$ denotes the number of vertices in $\bigsqcup_{z \in \Lambda(\sigma)}\mT_z$. In the last equality, we use the fact that the distribution of $S(\sigma)$ does not depend on $\sigma$. In the last expression, it follows from \eqref{eq:Rsigma_finite} that the first summand is finite. The finiteness of the second summand follows from the fact that $S(0_{\Z^d})$ has a finite $d$th moment. This completes the proof of \eqref{eq:sum_all_vertices}. 
\end{proof}


\begin{bibdiv}
\begin{biblist}


\bib{Simulation}{article}{
   author={Alm, Sven Erick},
   author={Deijfen, Maria},
   title={First passage percolation on $\Z^2$: a simulation study},
   journal={J. Stat. Phys.},
   volume={161},
   date={2015},
   number={3},
   pages={657--678},
   issn={0022-4715},
   review={\MR{3406703}},
   doi={10.1007/s10955-015-1356-0},
}

\bib{FPP_history}{book}{
   author={Auffinger, Antonio},
   author={Damron, Michael},
   author={Hanson, Jack},
   title={50 years of first-passage percolation},
   series={University Lecture Series},
   volume={68},
   publisher={American Mathematical Society, Providence, RI},
   date={2017},
   pages={v+161},
   isbn={978-1-4704-4183-8},
   review={\MR{3729447}},
}

\bib{Campanino}{article}{
   author={Campanino, M.},
   author={Russo, L.},
   title={An upper bound on the critical percolation probability for the
   three-dimensional cubic lattice},
   journal={Ann. Probab.},
   volume={13},
   date={1985},
   number={2},
   pages={478--491},
   issn={0091-1798},
   review={\MR{781418}},
}

\bib{Cox_Durrett}{article}{
   author={Cox, J. Theodore},
   author={Durrett, Richard},
   title={Some limit theorems for percolation processes with necessary and
   sufficient conditions},
   journal={Ann. Probab.},
   volume={9},
   date={1981},
   number={4},
   pages={583--603},
   issn={0091-1798},
   review={\MR{624685}},
}


\bib{1_dFPP}{article}{
   author = {Daniel, Ahlberg},
   title={Asymptotic of first-passage percolation on $1$-dimensional graphs},
   date={2016},
   Eprint = {arXiv:1107.2276}
}


\bib{Graph_theory}{book}{
   author={Diestel, Reinhard},
   title={Graph theory},
   series={Graduate Texts in Mathematics},
   volume={173},
   edition={5},
   note={Paperback edition of [ MR3644391]},
   publisher={Springer, Berlin},
   date={2018},
   pages={xviii+428},
   isbn={978-3-662-57560-4},
   isbn={978-3-662-53621-6},
   review={\MR{3822066}},
}

\bib{FKG}{article}{
   author={Fortuin, C. M.},
   author={Kasteleyn, P. W.},
   author={Ginibre, J.},
   title={Correlation inequalities on some partially ordered sets},
   journal={Comm. Math. Phys.},
   volume={22},
   date={1971},
   pages={89--103},
   issn={0010-3616},
   review={\MR{309498}},
}

\bib{Multidimensional_lattice}{article}{
   author={Grimmett, G. R.},
   title={Multidimensional lattices and their partition functions},
   journal={Quart. J. Math. Oxford Ser. (2)},
   volume={29},
   date={1978},
   number={114},
   pages={141--157},
   issn={0033-5606},
   review={\MR{489612}},
   doi={10.1093/qmath/29.2.141},
}

\bib{Grimmett}{book}{
   author={Grimmett, Geoffrey},
   title={Percolation},
   series={Grundlehren der Mathematischen Wissenschaften [Fundamental
   Principles of Mathematical Sciences]},
   volume={321},
   edition={2},
   publisher={Springer-Verlag, Berlin},
   date={1999},
   pages={xiv+444},
   isbn={3-540-64902-6},
   doi={10.1007/978-3-662-03981-6},
}

\bib{beyond_Z^d}{article}{
   author={H\"{a}ggstr\"{o}m, Olle},
   title={Percolation beyond $\Z^d$: the contributions of Oded Schramm},
   journal={Ann. Probab.},
   volume={39},
   date={2011},
   number={5},
   pages={1668--1701},
   issn={0091-1798},
   review={\MR{2884871}},
   doi={10.1214/10-AOP563},
}

\bib{Hammersley}{article}{
   author={Hammersley, J. M.},
   author={Welsh, D. J. A.},
   title={First-passage percolation, subadditive processes, stochastic
   networks, and generalized renewal theory},
   conference={
      title={Proc. Internat. Res. Semin., Statist. Lab., Univ. California,
      Berkeley, Calif.},
   },
   book={
      publisher={Springer-Verlag, New York},
   },
   date={1965},
   pages={61--110},
   review={\MR{0198576}},
}


\bib{Aspects}{article}{
   author={Kesten, Harry},
   title={Aspects of first passage percolation},
   conference={
      title={\'{E}cole d'\'{e}t\'{e} de probabilit\'{e}s de Saint-Flour, XIV---1984},
   },
   book={
      series={Lecture Notes in Math.},
      volume={1180},
      publisher={Springer, Berlin},
   },
   date={1986},
   pages={125--264},
   review={\MR{876084}},
   doi={10.1007/BFb0074919},
}

\bib{Kesten_textbook}{book}{
   author={Kesten, Harry},
   title={Percolation theory for mathematicians},
   series={Progress in Probability and Statistics},
   volume={2},
   publisher={Birkh\"{a}user, Boston, Mass.},
   date={1982},
   pages={iv+423},
   isbn={3-7643-3107-0},
   review={\MR{692943}},
}

\bib{Standard_Realizations}{article}{
   author={Kotani, Motoko},
   author={Sunada, Toshikazu},
   title={Standard realizations of crystal lattices via harmonic maps},
   journal={Trans. Amer. Math. Soc.},
   volume={353},
   date={2001},
   number={1},
   pages={1--20},
   issn={0002-9947},
   review={\MR{1783793}},
   doi={10.1090/S0002-9947-00-02632-5},
}

\bib{subadditive_ergodic}{article}{
   author={Liggett, Thomas M.},
   title={An improved subadditive ergodic theorem},
   journal={Ann. Probab.},
   volume={13},
   date={1985},
   number={4},
   pages={1279--1285},
   issn={0091-1798},
   review={\MR{806224}},
}

\bib{Strict_monotonicity}{article}{
   author={Martineau, S\'{e}bastien},
   author={Severo, Franco},
   title={Strict monotonicity of percolation thresholds under covering maps},
   journal={Ann. Probab.},
   volume={47},
   date={2019},
   number={6},
   pages={4116--4136},
   issn={0091-1798},
   review={\MR{4038050}},
   doi={10.1214/19-aop1355},
}

\bib{FPPonTriangular2019}{article}{
   author={Michael, Damron},
   author={Jack, Hanson},
   author={Wai-Kit, Lam},
   title={Universality of the time constant for $2D$ critical first-passage percolation},
   date = {2019},
   Eprint = {arxiv:1904.12009}
}



\bib{Topological_Crystallography}{book}{
   author={Sunada, Toshikazu},
   title={Topological crystallography},
   series={Surveys and Tutorials in the Applied Mathematical Sciences},
   volume={6},
   note={With a view towards discrete geometric analysis},
   publisher={Springer, Tokyo},
   date={2013},
   pages={xii+229},
   isbn={978-4-431-54176-9},
   isbn={978-4-431-54177-6},
   review={\MR{3014418}},
   doi={10.1007/978-4-431-54177-6},
}

\bib{FPPonTriangular2016}{article}{
   author={Yao, Chang-Long},
   title={Limit theorems for critical first-passage percolation on the
   triangular lattice},
   journal={Stochastic Process. Appl.},
   volume={128},
   date={2018},
   number={2},
   pages={445--460},
   issn={0304-4149},
   review={\MR{3739504}},
   doi={10.1016/j.spa.2017.05.002},
}

\end{biblist}
\end{bibdiv}
\end{document}